\documentclass{amsart}
\usepackage{amssymb,amsmath,amsthm,amsfonts}
\usepackage{graphicx}
\usepackage{esint}
\usepackage{tikz}
\usetikzlibrary{arrows.meta,positioning}
\usepackage{environ}
\usepackage{float}
\usepackage{xcolor}    

\usepackage{enumitem}

\DeclareMathOperator{\diam}{diam}
\DeclareMathOperator{\dist}{dist}

\newcommand{\field}[1]{\mathbb{#1}}
\newcommand{\N}{\field{N}}   
\newcommand{\Z}{\field{Z}} 
\newcommand{\R}{\field{R}}

\NewEnviron{bigequation}{
    \begin{equation}
    \scalebox{1.5}{$\BODY$}
    \end{equation}
    }

\newcommand{\eps}{\varepsilon}

\newcommand{\loc}{{\text{loc}}}

\def\Barint_#1{\mathchoice
          {\mathop{\vrule width 6pt height 3 pt depth -2.5pt
                  \kern -8pt \intop}\nolimits_{#1}}%
          {\mathop{\vrule width 5pt height 3 pt depth -2.6pt
                  \kern -6pt \intop}\nolimits_{#1}}%
          {\mathop{\vrule width 5pt height 3 pt depth -2.6pt
                  \kern -6pt \intop}\nolimits_{#1}}%
          {\mathop{\vrule width 5pt height 3 pt depth -2.6pt
                  \kern -6pt \intop}\nolimits_{#1}}}

\makeatletter
\newcommand*{\mint}[1]{%
  \mint@l{#1}{}%
}
\newcommand*{\mint@l}[2]{%
  \@ifnextchar\limits{%
    \mint@l{#1}%
  }{%
    \@ifnextchar\nolimits{%
      \mint@l{#1}%
    }{%
      \@ifnextchar\displaylimits{%
        \mint@l{#1}%
      }{%
        \mint@s{#2}{#1}%
      }%
    }%
  }%
}
\newcommand*{\mint@s}[2]{%
  \@ifnextchar_{%
    \mint@sub{#1}{#2}%
  }{%
    \@ifnextchar^{%
      \mint@sup{#1}{#2}%
    }{%
      \mint@{#1}{#2}{}{}%
    }%
  }%
}
\def\mint@sub#1#2_#3{%
  \@ifnextchar^{%
    \mint@sub@sup{#1}{#2}{#3}%
  }{%
    \mint@{#1}{#2}{#3}{}%
  }%
}
\def\mint@sup#1#2^#3{%
  \@ifnextchar_{%
    \mint@sup@sub{#1}{#2}{#3}%
  }{%
    \mint@{#1}{#2}{}{#3}%
  }%
}
\def\mint@sub@sup#1#2#3^#4{%
  \mint@{#1}{#2}{#3}{#4}%
}
\def\mint@sup@sub#1#2#3_#4{%
  \mint@{#1}{#2}{#4}{#3}%
}
\newcommand*{\mint@}[4]{
  \mathop{}%
  \mkern-\thinmuskip
  \mathchoice{%
    \mint@@{#1}{#2}{#3}{#4}%
        \displaystyle\textstyle\scriptstyle
  }{%
    \mint@@{#1}{#2}{#3}{#4}%
        \textstyle\scriptstyle\scriptstyle
  }{%
    \mint@@{#1}{#2}{#3}{#4}%
        \scriptstyle\scriptscriptstyle\scriptscriptstyle
  }{%
    \mint@@{#1}{#2}{#3}{#4}%
        \scriptscriptstyle\scriptscriptstyle\scriptscriptstyle
  }%
  \mkern-\thinmuskip
  \int#1%
  \ifx\\#3\\\else_{#3}\fi
  \ifx\\#4\\\else^{#4}\fi  
}
\newcommand*{\mint@@}[7]{%
  \begingroup
    \sbox0{$#5\int\m@th$}%
    \sbox2{$#5\int_{}\m@th$}%
    \dimen2=\wd0 %
    \let\mint@limits=#1\relax
    \ifx\mint@limits\relax
      \sbox4{$#5\int_{\kern1sp}^{\kern1sp}\m@th$}%
      \ifdim\wd4>\wd2 %
        \let\mint@limits=\nolimits
      \else
        \let\mint@limits=\limits
      \fi
    \fi
    \ifx\mint@limits\displaylimits
      \ifx#5\displaystyle
        \let\mint@limits=\limits
      \fi
    \fi
    \ifx\mint@limits\limits
      \sbox0{$#7#3\m@th$}%
      \sbox2{$#7#4\m@th$}%
      \ifdim\wd0>\dimen2 %
        \dimen2=\wd0 %
      \fi
      \ifdim\wd2>\dimen2 %
        \dimen2=\wd2 %
      \fi
    \fi
    \rlap{%
      $#5%
        \vcenter{%
          \hbox to\dimen2{%
            \hss
            $#6{#2}\m@th$%
            \hss
          }%
        }%
      $%
    }%
  \endgroup
}

\theoremstyle{definition}
\newtheorem{theorem}{Theorem}

\newtheorem{corollary}[theorem]{Corollary}
\newtheorem{lemma}[theorem]{Lemma}

\theoremstyle{definition}
\newtheorem{definition}[theorem]{Definition}

\newtheorem{remark}[theorem]{Remark}

\numberwithin{theorem}{section} \numberwithin{equation}{section}

\title{A purely metric  characterization of supercritical Sobolev spaces}
\date{\today}
\begin{document}
\author[Ryan Alvarado]{Ryan Alvarado}
\address{Department of Mathematics \\ Amherst College, Amherst, MA}
\email{rjalvarado@amherst.edu}
\author[Efstathios-K. Chrontsios-Garitsis]{Efstathios-K. Chrontsios-Garitsis}
\address{Department of Mathematics \\  The Ohio State University\\ 231 W 18th Ave\\ Columbus, OH 43210}
\email{chrontsios.1@osu.edu, echronts@gmail.com}

\maketitle

\begin{center}
\it Dedicated to Bob Kaufman, 1942–2025.
\end{center}

\begin{abstract}
In this paper we provide a purely metric and derivative-free characterization of continuous mappings in the local supercritical Sobolev space. More specifically, we show that a continuous mapping $f\colon\R^n\to\R^m$ lies in the local Sobolev space $W^{1,p}_\loc(\R^n\colon\R^m)$ if and only if $f$ is a $(p,1-n/p)$-compactly H\"older mapping, where $p>n$. In fact, we prove the following more general result: whenever $X$ is a complete $Q$-Ahlfors regular metric space supporting a $p$-Poincar\'e inequality with $p>Q$, and $V$ is any Banach space,  the compactly H\"older class $CH^{p,1-Q/p}(X\colon V)$ consists precisely of the continuous representatives of the local Haj\l asz and Newtonian Sobolev spaces $M^{1,p}_\loc (X \colon V)$ and $N^{1,p}_\loc(X\colon V)$. We also establish quantitative comparisons between the corresponding seminorms.
\end{abstract}

\section{Introduction}

Sobolev spaces are famously considered to be fundamental tools in modern analysis.  They provide the right framework to formulate weak notions of differentiability that are
stable under limiting procedures. The theory of Sobolev spaces has found successful applications within elliptic and parabolic partial differential equations (PDEs), nonlinear potential
theory, and the calculus of variations, to name a few (see for instance
\cite{Fract_Sob_book,EvansPDE,CalcVarGiaq,Heinonen2ndEdWeightedSobBook}).
The increasing interest and  need to extend this theory to metric spaces soon emerged, and has been a very active line of research for the past two and a half decades. Applications of this endeavor include the development of the theory of  PDEs \cite{KigamiAnFr}, calculus of variations \cite{AmbrosioBV} and optimal transportation \cite{AmbrosioGradFlow} on the nonsmooth setting of rough spaces. The theory of Sobolev-type mappings defined between metric spaces has been developed by many authors (for instance \cite{CheegerSob}, \cite{HajSob}, \cite{HajKoskSob}, \cite{KorSchSob}, \cite{NagesNewtonianSob}), who have used different ideas to adjust the theory to different settings. We refer to \cite{Juha-Jeremy-etc-book} for a detailed exposition.

Among the principal metric Sobolev notions
are the  Haj\l asz  and Newtonian
spaces, introduced by Haj\l asz in \cite{HajSob} and Shanmugalingam \cite{NagesNewtonianSob}, respectively.   These notions have proved effective precisely because they
replace classical derivatives by objects adapted to the geometry of the
underlying space.  They nevertheless retain measure-theoretic ties, through
integrability and almost-everywhere inequalities.
It is therefore natural to ask how much Sobolev regularity can be detected
using only the metric information of the underlying metric space $(X,d_X)$ and the target space $(Y,d_Y)$.  A Sobolev-type definition formulated solely in terms
of distances and balls has at least one advantage;  it is a well-defined notion before a measure is
chosen for the space $X$. This provides more insight into what a natural Sobolev structure might be for $X$, already in the Sobolev functions case.  The point is not to eliminate the measure from Sobolev
theory itself, but to isolate metric information which, once a suitable measure
is present,  is strong enough to recover the intrinsic  Sobolev
regularity of $X$.

The relevant Sobolev classes in this paper are two variations of the class of compactly H\"older
mappings, as introduced in \cite{Chron_comp_holder_Minkowski} and then modified in \cite{AC-G26}.
The main motivation behind the introduction of this notion was in connection with
the distortion of certain fractal dimensions. In particular, a question of broad interest has been to determine in what ways certain classes of mappings distort  dimension notions, a programme essentially initiated by Gehring-V\"ais\"al\"a \cite{GehringVais}. Kaufman in \cite{Kaufman} was the first one to introduce an appropriate 
stopping-time argument in this context, which later inspired an avalanche of results in this programme by many different authors that is very active (see for instance \cite{FraTyson_intermed_dim, SobSubcriticalDimDist1,SobSubcriticalDimDist2,OurQCspec,SobSubcritDimDist0,HolomSpecChron,btw:heisenberg,BaloghAGMS, BishHakWill:QSdistortion}).

Roughly speaking, a mapping  $f:X\to Y$  between two metric spaces belongs to the $(p,\alpha)$-compactly H\"older class
$CH^{p,\alpha}(X \colon Y)$ if, on every compact subset of $X$, the $p$-sum of its
local $\alpha$-H\"older coefficients is uniformly bounded over all sufficiently
fine ball covers $\{B(x_i,r)\}_{i\in I}$ that satisfy a given packing condition.
Thus, this notion combines a pointwise modulus of continuity
with an $\ell^p$ summability condition across location and scale.  Moreover, the definition involves the metrics of $X$ and $Y$, but no measure. The modified compactly H\"older class $CH^{p,\alpha}_m(X\colon Y)$ is defined similarly but for covers $\{B(x_i,r_i)\}_{i\in I}$ of potentially different radii $r_i$. 
Precise definitions are given in Section~2.

In the interest of comparing the $CH^{p,\alpha}(X \colon Y)$ and $CH^{p,\alpha}_m(X \colon Y)$ notions to other Sobolev notions, and specifically the Newtonian Sobolev maps, we focus on a complete $Q$-Ahlfors regular metric measure space
$(X,d,\mu)$ supporting a $p$-Poincar\'e inequality for $p>Q$ and we identify $Y$ with an arbitrary Banach space $V$. This is the natural setting for the Newtonian Sobolev class (see \cite{QCanalyticDefHKSTyson}). Note that the latter assumption on $Y$ is not an essential restriction due to the Kuratowski isometric embedding (see for instance \cite{Juha-Jeremy-etc-book}).  For $t>Q$, write
$$
    \alpha_t:=1-\frac{Q}{t}.
$$
This is the H\"older exponent for super-critical Sobolev maps resulting from the Morrey embedding (see for instance \cite{EvansPDE}).  Our first
main result shows that membership in the compactly-H\"older class at this exponent relation
recovers the corresponding Haj\l asz--Sobolev space and, consequently, the Newtonian Sobolev space, while in the modified compactly-H\"older case there is an arbitrarily small loss in the integrability exponent. Note that in the case of the usual Euclidean setting $X=\R^n$, $Y=\R^m$ the Haj\l asz and Newtonian notions both coincide with the classical Sobolev spaces $W^{1,p}_\loc(\R^n, \R^m)$, $p>n$. Therefore, the compactly-H\"older maps essentially provide a novel, purely metric  classification of super-critical Sobolev spaces in the Euclidean setting.

\begin{theorem}\label{thmA:almost-characterization}
Let $(X,d,\mu)$, $V$, $p$ be as above.  Then for every $q\in (Q,p)$, we have 
$$
\begin{aligned}
CH^{p,\alpha_p}_m(X \colon V)
\subset CH^{p,\alpha_p}(X \colon V)
= M^{1,p}_{\mathrm{loc}}(X \colon V)= N^{1,p}_{\mathrm{loc}}(X \colon V)
\subset 
CH^{q,\alpha_q}_m(X \colon V),
\end{aligned}
$$
    with the $M^{1,p}_{\text{loc}}(X \colon V)$, $N_\loc^{1,p}(X \colon V)$ spaces understood as the corresponding classes of continuous representatives.
\end{theorem}

The inclusion $CH^{p,\alpha_p}_m(X \colon V) \subset CH^{p,\alpha_p}(X \colon V)$ is immediate by the definition of the two notions, and the fact that $M^{1,p}_{\text{loc}}(X \colon V)=N_\text{loc}^{1,p}(X \colon V)$ in this context where $X$ is Ahlfors regular and supports a Poincar\'e inequality is standard. 
The restriction $q<p$ in the right-hand inclusion of Theorem~\ref{thmA:almost-characterization} arises from the use of the Hardy--Littlewood maximal function theorem to handle the required packing condition and the potentially distinct radii $r_i$ in the definition of $CH_m^{p,\alpha_p}(X\colon V)$ (see \cite{AC-G26}). While $q$ may be
chosen arbitrarily close to $p$, we show that even
in the Euclidean setting the remaining upper endpoint $q=p$ cannot, in general, be included (see Theorem~\ref{thm: W1p not CH}).  In
particular, for every $n\geq 1$ and $p>n$ we construct a continuous function
$u\in W^{1,p}(\R^n)\subset CH^{p,\alpha_p}(\R^n \colon \R)$ such that
$$
    u\notin CH_m^{p,\alpha_p}(\R^n \colon \R).
$$
Consequently,  the loss
$q<p$ for the modified compactly H\"older class in
Theorem~\ref{thmA:almost-characterization} is sharp in this sense.

There is also a quantitative local form of the result.  For a compact set
$K\subset X$ and $0<\eps<1$, let
$[f]_{p,\alpha_p,\eps}(K)$ denote the compactly H\"older seminorm obtained by
taking the supremum of the $p$-sums of the $\alpha_p$-H\"older coefficients
over admissible sufficiently $\eps$-fine covers of $K$.  Let
$\|f\|_{\dot{M}^{1,p}(A)}$ denote the infimum of the $L^p(A)$-norms of Haj\l asz upper
gradients of $f$ on $A$.

\begin{theorem}\label{thmB:seminorm-comparison}
Let $X, V, p$ be as in Theorem~\ref{thmA:almost-characterization}, and $0<\eps<1/8$.  Suppose that $K\subset X$ is compact,
$\Omega\subset X$ is open with $K\subset\Omega$, and $U\subset X$ is a
nonempty, connected, measurable set satisfying
$\dist(U,X\setminus K)>0.$
Then there exists a positive constant $C$ such that every continuous mapping $f:X\to V$ satisfies
$$
    C^{-1}\|f\|_{\dot{M}^{1,p}(U)}
    \leq
    [f]_{p,\alpha_p,\eps}(K)
    \leq
    C\|f\|_{\dot{M}^{1,p}(\Omega)}.
$$
\end{theorem}

The two inequalities in Theorem~\ref{thmB:seminorm-comparison} display the role
of the compactly H\"older seminorm  in a clean way.  On the one hand, it
controls a Sobolev
seminorm on an interior set that might be ``arbitrarily close'' to the set $K$.  On the other hand, it is controlled by a slightly
stronger Sobolev seminorm on a surrounding open set that is also arbitrarily close.  Thus, it provides a purely metric description of local Sobolev regularity, with quantitative comparisons on nested sets. It should be noted that we also prove a version of the above result for arbitrary $U$. However, in that case the trade-off is that $U$ has to be contained uniformly deep in the interior of $K$ (see Theorem~\ref{thm: CH-Haj seminorms - OLD}).

It should be noted that in variational and PDE
settings, the  inequality in Theorem~\ref{thmB:seminorm-comparison} gives
a derivative-free way to certify Sobolev regularity of a candidate map or
limit on interior regions, a precious trait for metric-valued variational problems
\cite{CalcVarGiaq,KorSchSob}.  In addition,
Sobolev maps  have been recently used in machine learning and AI both to match
derivatives of target functions and to regularize adversarial critics
\cite{AlSobolevTraining, AlSobolevGAN}. Since the compactly
H\"older functional is built purely from distances, it
suggests a derivative-free regularity alternative for nonsmooth 
data that could potentially decrease computation costs.

We briefly outline the arguments behind the two main results.  Starting with a
compactly H\"older map, we construct Lipschitz approximations by means of
partitions of unity and local Bochner averages.  The defining packing condition for compactly H\"older maps provides
 a uniform $L^p$ bound for the pointwise Lipschitz constants.  The
Poincar\'e inequality and the Hardy--Littlewood maximal operator then produce
uniformly bounded Haj\l asz  gradients, and a weak compactness argument involving Mazur's lemma passes these to the limit.  The converse implication $M^{1,p}_\loc(X \colon V)\subset CH^{p,\alpha_p}(X \colon V)$ follows by taking advantage of the Ahlfors regularity of $X$, the uniformly bounded overlap in the packing condition of the compactly H\"older definition and the identification of the Haj\l asz and Newtonian spaces in this context. For the inclusion $M^{1,p}_\loc(X \colon V)\subset CH_m^{q,\alpha_q}(X \colon V)$ we no longer have a uniformly bounded overlap for admissible coverings. Instead, we apply the metric Morrey estimate we proved in \cite{AC-G26} to  bound each local
H\"older coefficient by an integral of a Sobolev upper gradient.  The pairwise
disjoint contracted balls allow these bounds to be summed, while the strong
maximal inequality at exponent $p/q>1$ accounts for the restriction $q<p$.
Moreover, the Euclidean counterexample concentrates disjoint smooth bumps at rapidly
decreasing scales, while maintaining their Sobolev energies summable. However, the construction ensures that certain admissible balls for the modified compactly H\"older definition
contain their endpoint H\"older oscillations often enough to force a divergent
$p$-sum.

The paper is organized as follows.  Section~2 collects the background on
Sobolev spaces, Poincar\'e inequalities, Morrey estimates, and
compactly H\"older mappings.  In Section~3 we construct the necessary Lipschitz
approximations to prove 
Theorem~\ref{thmA:almost-characterization}.
In Section~4 we
introduce the localized compactly H\"older seminorms and prove
Theorem~\ref{thmB:seminorm-comparison},
including the corresponding comparison on balls.
Section~5 gives the construction of a Euclidean super-critical Sobolev function of integrability $p$ that does not lie in the modified compactly H\"older class of the same exponent.  

\vspace{0.5cm}
\paragraph{\bf Acknowledgements.} The second author wishes to thank Vyron Vellis and Piotr Hajłasz for the valuable discussions on the topic. The authors wish to dedicate the manuscript to Bob Kaufman, especially since his covering and stopping-time argument in \cite{Kaufman} is what essentially led to the definition of the compactly H\"older spaces. Lastly, the second author is partially supported by an AMS-Simons Travel Grant.

\section{Background}
A triplet $(X,d,\mu)$ is called a \textit{metric measure space}  if $(X,d)$ is a metric space and $\mu$ is a nonnegative Borel measure on $X$ that assigns a strictly positive and finite value on all balls in $X$. Throughout this paper, all measures are considered to have the aforementioned properties, even if not stated explicitly. Note that every metric measure space is necessarily separable (see \cite{Gorka21}).

The following notions for measures are also typically assumed in this setting.  
We say that $(X,d,\mu)$ is \textit{$Q$-Ahlfors regular} for some $Q\geq1$,
if there is a constant $C_\mu\geq 1$ such that
$$
C_\mu^{-1}r^Q \leq \mu (B(x,r)) \leq C_\mu r^Q,
$$ for all $x\in X$ and $r\in (0,\diam X)$.
We say a metric space $(X,d)$ is \textit{$Q$-Ahlfors regular} if there is a measure $\mu$ on $X$ such that $(X,d,\mu)$  is $Q$-Ahlfors regular. Recall that if $X$ is Ahlfors regular, then it is a \textit{doubling} metric space, i.e., there is some $N_X\in \N$ such that every ball $B(x,r)$ can be covered by at most $N_X$ balls of radius $r/2$.

For $p\in (0,\infty]$ we denote by $L^p(X\colon V)$ the space of \textit{$p$-integrable} mappings $f\colon X\to V$ in the Bochner sense, and by $L^p_\loc(X\colon V)$ the corresponding space of locally $p$-integrable mappings. Moreover, for a ball $B\subset X$ and $f\in L^1(B\colon V)$ we denote by $f_B$ the Bochner average of $f$ over $B$, i.e., $f_B:= \fint_{B}f\, d\mu = \mu(B)^{-1} \int_{B} f\, d\mu$. If $V=\R$, we write $L^p(X)=L^p(X\colon \R)$ and $L^p_\loc (X)=L^p_\loc(X\colon \R)$. For more details on Lebesgue theory of Banach space valued functions, see \cite[Ch.~3]{Juha-Jeremy-etc-book}.

For $g\in L^1_{\text{loc}}(X)$, define the Hardy--Littlewood maximal function
$$
Mg(x):=
\sup_{r>0}\fint_{B(x,r)} |g|\,d\mu,
$$
and, for $s>0$,
$$
M_sg(x):=\bigl(M(|g|^s)(x)\bigr)^{1/s}.
$$
We recall the standard Hardy--Littlewood maximal theorem on measure doubling spaces (see Section~3.5 in \cite{Juha-Jeremy-etc-book}) for $1<p<\infty$:
$$
\|Mg\|_{L^p(X)}\lesssim\|g\|_{L^p(X)}.
$$
Consequently, if $p>s$, then
\begin{equation}
\|M_sg\|_{L^p(X)}\label{eq: maxim Mq ineq}
=
\|M(|g|^s)\|_{L^{p/s}(X)}^{1/s}
\le C\|g\|_{L^p(X)}.
\end{equation}

\subsection{Sobolev-type mappings}\label{subsec:BackSobolev}
Let $(X,d_X)$ and $(Y,d_Y)$ be two metric spaces.
Given $\alpha\in (0,\infty)$, a mapping $f:X\rightarrow Y$ and a set $B\subset X$, we  define the \textit{$\alpha$-H\"older coefficient} of $f$ on $B$ as
$$
|f|_{\alpha, B}:= \sup\left\{ \frac{d_Y(f(x),f(y))}{[d_X(x,y)]^\alpha}: \, x, y \in B \,\, \text{distinct} \right\}.
$$ If $|f|_{\alpha, B}<\infty$ then we say that $f$ is \textit{$\alpha$-H\"older continuous} in $B$. 

Given an at most countable index set $I$, we denote by $\ell^p(I)$, $p\in (0,\infty)$, the space of real-valued sequences $\{c_i\}_{i\in I}$ with finite $p$-norm $(\sum_{i\in I} c_i^p)^{1/p}<\infty$. We call $\sum_{i\in I} c_i^p$ the \textit{$p$-sum} of the sequence $\{c_i\}_{i\in \N}$.

For the rest of the paper, all index sets are assumed to be at most countable. We now recall the classes of compactly H\"older and  modified compactly H\"older mappings.

\begin{definition}\label{def: CH maps}
    Let $f:X\rightarrow Y$ be a mapping between two arbitrary metric spaces. For $p, \alpha\in(0,\infty)$, we say $f$ is \textit{$(p,\alpha)$-compactly H\"older}, and write $f\in CH^{p,\alpha}(X\colon Y)$, if for any compact set $E\subset X$ and any $\eps\in (0,1)$ there are $r_E>0$ and $C_E>0$  satisfying the following: \\
if $\{B_i\}_{i\in I}$ is a collection of balls $B_i:=B(x_i,r)$ with $x_i\in X$, $r<r_E$ that covers $E$ and $B(x_i,\eps r)\cap B(x_j,\eps r)=\emptyset$ for all distinct $i, j\in I$, then the $p$-sum of the H\"older coefficients of $f$ on $B_i$ is at most $C_E$, i.e.,
\begin{equation}\label{eq: CH-def-inequality}
    \sum\limits_{i\in I} |f|_{\alpha, B_i}^p\leq C_E.
\end{equation} Moreover, we say that $f$ is a \textit{modified $(p,\alpha)$-compactly H\"older}, and write $f\in CH_m^{p,\alpha}(X\colon Y)$, if the above also  holds for coverings of the form $B_i:=B(x_i,r_i)$ with $r_i<r_E$ and $B(x_i,\eps r_i)\cap B(x_j,\eps r_j)=\emptyset$ for all distinct $i, j\in I$.
\end{definition}

Here we follow the convention that if $\{B(x_i,r_i)\}_{i\in I}$ covers $E$, it is implied that $B(x_i,r_i)\cap E\neq\emptyset$ for all $i$, but not all $x_i$  necessarily lie in $E$. Note that  applying the definition on singleton sets yields that compactly H\"older mappings are H\"older continuous on compact sets. 

\begin{remark}\label{rem: old def}
    In \cite{Chron_comp_holder_Minkowski} the class of compactly H\"older maps was initially defined by the second author with $\{B(x_i,r)\}$ being a cover where all balls are of the same radius $r$ (see \cite[Definition 2.3]{Chron_comp_holder_Minkowski}). While that condition is enough to determine the distortion of the Minkowski dimension, it is generally not enough to study finer notions of fractal dimension, such as the intermediate dimensions \cite{Intermediate_dim_introduction}, unless additional regularity is assumed for the underlying space $X$. As a result, while working on the distortion of the intermediate dimensions, the authors in \cite{AC-G26} focused on the modified class $CH_m^{p,\alpha}(X \colon Y)$ and called that  compactly H\"older for simplicity, especially due to how recent the actual compactly H\"older $CH^{p,\alpha}(X \colon Y)$ notion was at the time, with the equivalence to Sobolev mappings being unclear. We henceforth distinguish the two notions, starting in this manuscript. 
\end{remark}

Let $(X,d,\mu)$ be a metric measure space and $(V,\|\cdot \|)$ be an arbitrary Banach space.
Following \cite{HajSob,hajlasz,Y03},
a measurable function  $g:X\to[0,\infty]$  is called \textit{Haj\l{}asz  gradient} of a measurable function $f\colon  X\rightarrow V$ if
there exists a set $E\subset X$ with $\mu(E)=0$ such that
\begin{equation}
\label{fracHajlasz}
\|f(x)-f(y)\|\leq d(x,y)\left[g(x)+g(y)\right],
\end{equation}
for every $x,y\in X\setminus E$.
The collection of all the  Haj\l{}asz gradients of $f$ is denoted by $D(f)$.
Given $p\in(0,\infty)$, the \textit{homogeneous Haj\l asz--Sobolev space}
$\dot{M}^{1,p}(X\colon V)$ is defined as the collection of
all the measurable functions $f\colon  X\rightarrow V$ such that
\begin{equation*}
\Vert f\Vert_{\dot{M}^{1,p}(X\colon V)}:=
\inf_{g\in D(f)}\Vert g\Vert_{L^p(X)}<\infty.
\end{equation*}
Here and thereafter, we make the agreement that for seminorms $\inf\emptyset:=\infty$.
The \textit{inhomogeneous Haj\l asz--Sobolev space} is defined as
$${M}^{1,p}(X\colon V):=\dot{M}^{1,p}(X\colon V)\cap L^p(X\colon V),$$ equipped with the `norm'
\begin{equation*}
\Vert f\Vert_{{M}^{1,p}(X \colon V)}:=\Vert f\Vert_{L^p(X\colon V)}+
\Vert f\Vert_{\dot{M}^{1,p}(X\colon V)}<\infty.
\end{equation*}

The following Morrey estimate was proven in \cite[Corollary~4.6 and Remark~4.7]{AC-G26}.
\begin{lemma}
\label{DOUBembedding-cpt}
Let $(X,d,\mu)$ be a metric measure space, where $\mu$ is  $Q$-Ahlfors regular for some $Q>0$, and let $(V,\|\cdot \|)$ be any Banach space. Suppose $p>Q$
and assume $f:X\to V$ is a continuous function. 
Then, for any compact set $K\subset X$, there exist constants $C\geq1$ and $r_0>0$, both of which are independent of $f$, such that, for all 
balls $B:=B(z,r)$ with $z\in K$ and
$0<r<r_0$ for which $\Vert f\Vert_{\dot{M}^{1,p}(2B\colon V)}<\infty$, one has
\begin{equation}
\label{eq30-DOUB-cpt}
\|f(x)-f(y)\|\leq C[d(x,y)]^{1-Q/p}\,\Vert f\Vert_{\dot{M}^{1,p}(2B\colon V)},
\end{equation}
for all $x,y\in B$. 
\end{lemma}

We now turn to the definition of the Newtonian Sobolev spaces. We say that a Borel function
$g:X\to[0,\infty]$ is an \textit{upper gradient} of a map
$f:X\to V$ if for every rectifiable curve
$\gamma:[0,1]\to X$ we have
\[
    \|f(\gamma(0))-f(\gamma(1))\|
    \leq
    \int_\gamma g\,ds.
\]
The notion was introduced by Heinonen and Koskela in \cite{HeinKoskQCbegin}
under a different name, and was employed by Shanmugalingam in
\cite{NagesPHDNewtonianSob} and \cite{NagesNewtonianSob} in order to define an appropriate notion of
Sobolev mappings in the metric measure space context. More specifically,
for $p>1$,  the \textit{Newtonian--Sobolev space}
$N^{1,p}(X\colon V)$ is defined as the collection of equivalence classes
of mappings $f:X\to V$ in $L^p(X\colon V)$ with an upper gradient in
$L^p(X)$. In fact, it is enough to only consider $p$-weak upper gradients (see \cite[Lemma~6.2.2]{Juha-Jeremy-etc-book}).

The local counterpart of $M^{1,p}(X\colon V)$, denoted by $M^{1,p}_{\loc}(X\colon V)$, is the collection of equivalence classes
of mappings $f\colon X\to V$ that lie in $M^{1,p}(K\colon V)$ for every compact $K\subset X$. The local counterpart of $N^{1,p}(X\colon V)$, denoted by $N^{1,p}_{\loc}(X\colon V)$, is defined accordingly.

We say that $X$ supports an \textit{$s$-Poincar\'e inequality} for some $s\geq1$ if, there exist $C_P>0$ and $\lambda\geq1$ such that, for every Lipschitz function $u:X\rightarrow\R$, we have
\begin{equation}\label{eq: PI}
\fint_{B(z,r)} |u-u_{B(z,r)}|\,d\mu
\le
C_P r
\left(
\fint_{B(z,\lambda r)}
(\operatorname{Lip}u)^s\,d\mu
\right)^{1/s}
\end{equation}
for all balls $B(z,r)\subset X$, 
where
$$\operatorname{Lip}u(x):=\limsup_{r\rightarrow0} \sup_{y\in B(x,r)}\frac{|u(x)-u(y)|}{r}$$ 
denotes the upper pointwise Lipschitz constant of $u$;  see \cite[Theorem 8.4.2]{Juha-Jeremy-etc-book}. We call the constant $\lambda$ the \textit{Poincar\'e radial constant}. 

Due to Keith--Zhong \cite{OpenPoincareKeith},  since $X$ is complete, doubling, and supports a $p$-Poincar\'e inequality, we have that $X$ supports a $s$-Poincar\'e inequality for some $1\leq s<p$. We fix such $s\in[1,p)$ for the rest of the manuscript. 


\section{Proof of characterization of Sobolev mappings}

Throughout this section we assume $(X,d,\mu)$ is a complete  $Q$-Ahlfors regular metric measure space supporting a $p$-Poincar\'e inequality with $p>Q\geq1$. We first provide a standard Lipschitz partition of unity with appropriate control of the pointwise Lipschitz constant.

\begin{lemma}\label{le: partition of unity}
Let $f\in L^1_{\text{loc}}(X\colon V)$. For every $t>0$, there exists $\Phi_tf:X\rightarrow V$ with
$$
\Phi_t f(x)=\sum_i \phi_i(x) f_{B(x_i,t)},
$$
where $\{x_i\}_i$ is a $t$-net, the balls $B(x_i,t/2)$ are pairwise disjoint, and $\{\phi_i\}_i$ is a Lipschitz partition of unity with
$$
0\le \phi_i\le 1,\qquad 
\operatorname{supp}\phi_i\subset B(x_i,2t),
\qquad
\sum_i\phi_i\equiv 1,
\qquad
\operatorname{Lip}\phi_i\lesssim t^{-1}.
$$
Moreover, $\Phi_t f$ is locally Lipschitz (i.e. Lipschitz on compact sets) and there is a constant $A\ge 1$, depending only on the doubling constant, such that
$$
\operatorname{Lip}(\Phi_t f)(x)
\lesssim
\frac{1}{t}
\fint_{B(x,At)}
\|f-f_{B(x,At)}\|\,d\mu.
$$
\end{lemma}

\begin{proof}
Fix a maximal $t$-separated set $\{x_i\}_i$. Then $X=\bigcup_i B(x_i,t)$, while the balls $B(x_i,t/2)$ are pairwise disjoint. 

Fix the Lipschitz cutoff functions 
$\widetilde\phi_i:X\rightarrow\R$
with
$$
\widetilde{\phi}_i(x)= \max \left\{0, \min\left\{1, 2-\frac{d(x,x_i)}{t}\right\}\right\}, \qquad\text{for all} \,\, x\in X.
$$
This choice yields
$$
0\le \widetilde\phi_i\le 1,\qquad
\widetilde\phi_i\equiv 1 \text{ on } B(x_i,t),
\qquad
\operatorname{supp}\widetilde\phi_i\subset B(x_i,2t),
\qquad
\operatorname{Lip}\widetilde\phi_i\leq t^{-1}.
$$
Then
$$
1\le \sum_j \widetilde\phi_j\le C,
$$
and we set
$$
\phi_i=\frac{\widetilde\phi_i}{\sum_j\widetilde\phi_j}.
$$
This gives the desired partition of unity $\{\phi_i\}_{i\in I}$. It remains to show that $\Phi_t f(x)=\sum_i \phi_i(x) f_{B(x_i,t)}$ is locally Lipschitz, and the desired upper bound on $\operatorname{Lip} (\Phi_tf)$.

Let $K\subset X$ be a compact set. Then the set of indices
$$
I_K:=\left\{i\in I: B(x_i,2t)\cap K\neq\emptyset \right\}
$$ has finitely many elements. 
Thus, for any $x,y\in K$, we have
$$
\|\Phi_tf(x)-\Phi_t f(y)\|\leq \sum_{i\in I_K} \|f_{B(x_i,t)}\|\,|\phi_i(x)-\phi_i(y)|\lesssim t^{-1} d(x,y),
$$ by the Lipschitz continuity of $\phi_i$, and with $C(\lesssim)$ depending only on $K$ and $t$.

Let $x,y\in X$ with $d(x,y)\le t$. Using $\sum_i\phi_i\equiv 1$ to write 
$$
\Phi_tf(x)=\sum_i \phi_i(x)\sum_j \phi_j(y) f_{B(x_i,t)}=\sum_{i,j} \phi_i(x) \phi_j(y) f_{B(x_i,t)},
$$ and similarly for $\Phi_tf(y)$,
we have
\begin{equation}\label{eq: diff of Phi}
\Phi_t f(x)-\Phi_t f(y)
=
\frac12\sum_{i,j}
\bigl(\phi_i(x)\phi_j(y)-\phi_j(x)\phi_i(y)\bigr)(f_{B(x_i,t)}-f_{B(x_j,t)}).
\end{equation}
Only uniformly many indices $i,j$ contribute, and for all such indices the balls $B(x_i,t)$ and $B(x_j,t)$ are contained in $B(x,At)$, for a uniform constant $A$.
By Ahlfors regularity, these balls have measure comparable to $\mu(B(x,At))$, up to a constant depending only on $A$ and the doubling constant. Hence, for every contributing $i\in I$,
$$
\begin{aligned}
\left\| f_{B(x_i,t)} - f_{B(x,At)} \right\|
&= \left\| \frac{1}{\mu(B(x_i,t))} 
    \int_{B(x_i,t)} \bigl(f(z)-f_{B(x,At)}\bigr)\, d\mu(z) \right\| \\
&\leq \frac{\mu(B(x,At))}{\mu(B(x_i,t))}
    \fint_{B(x,At)}
    \left\|f-f_{B(x,At)}\right\|\, d\mu \\
&\lesssim
    \fint_{B(x,At)}
    \left\|f-f_{B(x,At)}\right\|\, d\mu .
\end{aligned}
$$
The above yields that
\begin{equation}\label{eq: diff of avg f balls}
\begin{split}
\|f_{B(x_i,t)}-f_{B(x_j,t)}\|
&\le \|f_{B(x_i,t)}-f_{B(x,At)}\|+\|f_{B(x_j,t)}-f_{B(x,At)}\|
\\
&\lesssim
\fint_{B(x,At)}
\|f-f_{B(x,At)}\|\,d\mu.
\end{split}
\end{equation}
On the other hand, by $\operatorname{Lip}\phi_i\lesssim t^{-1}$ and $\phi_i\leq 1$, we have
\begin{equation}\label{eq: diff prod phi}
|\phi_i(x)\phi_j(y)-\phi_j(x)\phi_i(y)|\leq|\phi_j(y)||\phi_i(x)-\phi_i(y)|+|\phi_i(y)||\phi_j(x)-\phi_j(y)|
\lesssim
\frac{d(x,y)}{t}.
\end{equation}
Hence, by \eqref{eq: diff of Phi}, \eqref{eq: diff of avg f balls}, and \eqref{eq: diff prod phi} we have
$$
\|\Phi_t f(x)-\Phi_t f(y)\|
\lesssim
\frac{d(x,y)}{t}
\fint_{B(x,At)}
\|f-f_{B(x,At)}\|\,d\mu.
$$
Using the definition of $\operatorname{Lip}\Phi_t f$ and taking $r<t$ yields the desired upper bound due to the above inequality.
\end{proof}

We aim to apply the above partition of unity on compactly H\"older mappings, and so in the next lemma we compare the $L^p$-norm of the resulting pointwise Lipschitz constant with the compactly H\"older constant $C_K$ from Definition~\ref{def: CH maps}. This is essential in achieving the finiteness of the corresponding upper gradients in later arguments.

\begin{lemma}\label{le: compHold gives Lip in Lp}
Suppose $f\in CH^{p, \alpha_p}(X \colon V)$ and $K_0\subset X$ is compact. Then there exist $t_0>0$ and $C_{K_0}<\infty$ such that, for every $0<t<t_0$,
$$
\int_{K_0} \bigl(\operatorname{Lip}\Phi_t f\bigr)^p\,d\mu
\le C_{K_0}.
$$
\end{lemma}

\begin{proof}
Fix $t>0$.  By Lemma \ref{le: partition of unity} we have for all $x\in K_0$ that
$$
\operatorname{Lip}(\Phi_t f)(x)
\lesssim
\frac1t
\fint_{B(x,At)}
\|f-f_{B(x,At)}\|\,d\mu.
$$
Set $B=B(x,At)$ and note that
$$
\begin{aligned}
\fint_B \|f(x)-f_B\|\,d\mu(x)
&=
\fint_B
\left\|
\fint_B (f(x)-f(y))\,d\mu(y)
\right\|
d\mu(x)
\\
&\leq
\fint_B\fint_B \|f(x)-f(y)\|\,d\mu(y)\,d\mu(x)
\\
&\leq
\sup_{x,y\in B} \|f(x)-f(y)\|
\\
&\leq (2At)^{\alpha_p} |f|_{{\alpha_p},B(x,At)}.
\end{aligned}
$$
Therefore
$$
\operatorname{Lip}(\Phi_t f)(x)
\lesssim
t^{-1}t^{\alpha_p} |f|_{{\alpha_p},B(x,At)}
=
t^{-Q/p}|f|_{{\alpha_p},B(x,At)},
$$
due to $1-{\alpha_p}=Q/p$.
Raising to the power $p$ and integrating over $K_0$ yields
\begin{equation}\label{eq: Lp norm of Lip with Holder}
\int_{K_0} \bigl(\operatorname{Lip}\Phi_t f\bigr)^p\,d\mu
\lesssim
\int_{K_0} t^{-Q}|f|_{{\alpha_p},B(x,At)}^p\,d\mu(x).
\end{equation} It remains to show that the right-hand side of \eqref{eq: Lp norm of Lip with Holder} is bounded.

Fix a maximal $t$-separated set $\{z_j\}_j\subset K_0$. Then
$$
K_0\subset \bigcup_j B(z_j, t),
$$
while the balls $B(z_j, t/2)$ are pairwise disjoint.

Fix a measurable partition $\{E_j\}_j$ of $K_0$ with $E_1=K_0\cap B(z_1,t)$ and 
$$ 
E_j= (K_0\cap B(z_j,t))\setminus\cup_{i<j} B(z_i,t),
$$ for all $j\geq 2$. This choice ensures that $E_j\subset B(z_j, t)$.
If $x\in E_j$, then $B(x,At)\subset B(z_j,(A+1)t)$.
By choice of $E_j$ and monotonicity of the H\"older seminorm with respect to the ball, i.e.,
$$|f|_{{\alpha_p},B(x,At)}\le|f|_{{\alpha_p},B(z_j,(A+1)t)},$$ we have
\begin{equation}\label{eq: lem bound by Holder semin}
\begin{aligned}
\int_{K_0}
t^{-Q}|f|_{{\alpha_p},B(x,At)}^p\,d\mu(x)
&\le
\sum_j
\int_{E_j}
t^{-Q}|f|_{{\alpha_p},B(z_j,(A+1)t)}^p\,d\mu(x)
\\
&\lesssim
\sum_j |f|_{{\alpha_p},B(z_j,(A+1)t)}^p,
\end{aligned}
\end{equation}also due to the $Q$-Ahlfors regularity $\mu(E_j)\le \mu(B(z_j, t))\lesssim t^Q$.

Set
$$
B_j:=B(z_j,(A+1)t) \qquad\text{and} \qquad \varepsilon:=\frac{1}{4(A+1)}
$$
The balls $B_j$ cover $K_0$, while
$$
B(z_j,\varepsilon(A+1)t)=B(z_j, t/4)
$$
are pairwise disjoint, since the centers $z_j$ are $t$-separated. For $t>0$ sufficiently small, say $t<t_0$, all radii $(A+1)t$ are less than the scale $r_{K_0}$ in the compactly H\"older condition for the compact set $K_0$ and the scale $r_\mu$ from the Ahlfors regularity condition. Therefore,
$$
\sum_j |f|_{{\alpha_p},B_j}^p\le C_{K_0}.
$$
Since $A$ and $\varepsilon$ are fixed, so are $t_0$ and  $C_{K_0}$. The above inequality along with \eqref{eq: Lp norm of Lip with Holder} and \eqref{eq: lem bound by Holder semin} prove the claim.
\end{proof}

We are now preparing for the inclusion of compactly H\"older mappings in the Haj\l asz Sobolev space. Thus, we need the following lemma on the appropriate Haj\l asz gradient for a Lipschitz map, which we plan on applying to the aforementioned partition of unity map. 

\begin{lemma}\label{lem: Haj grad for Lip}
Let $U\subset X$ be a nonempty, connected, measurable set, and let $K\subset X$ be compact with
$\dist(U,X\setminus K)
>0.$ 
Suppose $u:X\to V$ is Lipschitz on a neighborhood of $K$, and let $s\geq1$ be as in the Poincar\'e inequality \eqref{eq: PI}. Then there is a constant $C>0$, such that
$$
h_u:=C\,M_s\bigl(\chi_K\operatorname{Lip}u\bigr)
$$
is a Haj\l asz  gradient of $u$ on $U$, i.e.,
\begin{equation}
\label{eq: Banach maximal fun upper grad in nbhd}
    \|u(x)-u(y)\|
\le
d(x,y)\bigl(h_u(x)+h_u(y)\bigr)
\end{equation}
for a.e. $x,y\in U$.
\end{lemma}

\begin{proof}
We begin with several observations. We first show that it is enough to prove \eqref{eq: Banach maximal fun upper grad in nbhd} for all real valued  functions that are Lipschitz in a neighborhood of $K$. Indeed, suppose that \eqref{eq: Banach maximal fun upper grad in nbhd} holds true for all $u:X\to \R$ that are Lipschitz in a neighborhood of $K$, and let $w:X\to V$ be Lipschitz in a neighborhood of $K$. Fix a dense subset $\{u_j:j\in \N\}$ of $w(K)$, and define for every $j\in \N$ a function $w_j:X\to \R$ with
$$
w_j(x)=\| w(x)-u_j\|,
$$ for all $x\in X$. Then $w_j$ is Lipschitz in the same neighborhood of $K$ as $w$, with $\operatorname{Lip} w_j\leq \operatorname{Lip} w$ for all $j\in \N$. Since $w_j$ are real valued, we have
\begin{equation}\label{eq: char upper grad for real valued}
    |w_j(x)-w_j(y)|\lesssim d(x,y)(h_{w_j}(x)+h_{w_j}(y))
\end{equation} for all $x,y\in U\setminus \cup_j N_j$ with $\mu(N_j)=0$ and all $j\in \N$. Note that the constant $C(\lesssim)$ above is independent of $j$. 

Let $x,y\in U \setminus \cup_j N_j$. By density of $\{u_j:j\in \N\}$ in $w(K)$, there is a sequence $\{u_{j_n}\}$ that converges to $w(x)$. Due to 
$$
| \|w(x)-u_{j_n}\|-\|w(y)-u_{j_n}\||\leq \|w(x)-w(y)\|
$$ and
\begin{align*}
\left| \|w(x)-u_{j_n}\|-\|w(y)-u_{j_n}\|\right|&\geq \| w(y)-u_{j_n}\|-\|w(x)-u_{j_n}\| \\ &\geq \|w(x)-w(y)\|-2\|w(x)-u_{j_n}\|
\end{align*}
we have that
$$
\lim_{n\to \infty} | \|w(x)-u_{j_n}\|-\|w(y)-u_{j_n}\||= \|w(x)-w(y)\|.
$$ Thus,
$$
\| w(x)-w(y)\|\leq \sup_j | \|w(x)-u_{j}\|-\|w(y)-u_{j}\||,
$$ which by choice of $w_j$, and the inequalities \eqref{eq: char upper grad for real valued} and $\sup_j h_{w_j}\leq h_w$ implies the desired inequality for the Banach valued map $w$.
Thus, for the rest of the proof we assume without loss of generality that $u:X\to\R$. 

We also note that it is enough to prove the estimate for points $x,y\in U$ for which the maximal functions below are finite. This holds for a.e. $x,y\in U$. Fix such $x,y\in U$. Since $u$ is continuous, every point is a Lebesgue point of $u$ (see for instance \cite[p.~77]{Juha-Jeremy-etc-book}). 

We can also assume $\diam U>0$ since otherwise $U$ would consist of only one point and there is nothing to prove. Note that $U\subset K$, where $K$ is compact, and so $\diam U<\infty$. Furthermore, we may assume without loss of generality that $u$ is globally Lipschitz. Indeed, by assumption, there exists an open set $V\supset K$ such that $u|_V$ is Lipschitz. Appealing to the McShane extension theorem (see \cite{mcshane}) gives us that $u|_V$ has a globally Lipschitz extension to $X$. Since replacing $u$ by this extension does not change either $u$ or $\operatorname{Lip}u$ on $V$, the estimates below are not affected, and the claim follows.

Next, we establish a point-to-average estimate. Specifically, we claim that there exists a constant $C=C(\mu,PI)>0$ such that 
\begin{equation}\label{eq:point-average-estimate}
|u(z)-u_{B(z,r)}|
\leq
Cr\,M_s\bigl(\chi_K\operatorname{Lip}u\bigr)(z),
\end{equation}
for all $z\in U$ and $r>0$ satisfying $B(z,\lambda r)\subset K$, where $\lambda\geq1$ is the Poincar\'e radial constant. Fix  $z\in U$ and $r>0$ satisfying $B(z,\lambda r)\subset K$, and for $k\in \Z$ non-negative, set 
$$
B_k^z:=B(z,2^{-k}r).
$$
Since $z$ is a Lebesgue point of $u$,
$$
u_{B_k^z}\to u(z)
\qquad\text{as }k\to\infty .
$$
Hence, by a finite telescoping argument and triangle inequality,
$$
\begin{aligned}
|u(z)-u_{B_0^z}|
&\le
\sum_{k=0}^{\infty}
|u_{B_{k+1}^z}-u_{B_k^z}|.
\end{aligned}
$$
Because $B_{k+1}^z\subset B_k^z$ and the radii differ by a factor $2$, Ahlfors regularity gives
$$
\begin{aligned}
|u_{B_{k+1}^z}-u_{B_k^z}|
&\le
\fint_{B_{k+1}^z}|u-u_{B_k^z}|\,d\mu
\\
&\lesssim
\fint_{B_k^z}|u-u_{B_k^z}|\,d\mu,
\end{aligned}
$$
which along with the previous inequality imply
\begin{equation}\label{eq: u minus avg est}
|u(z)-u_{B_0^z}|
\lesssim
\sum_{k=0}^{\infty}
\fint_{B_k^z}
|u-u_{B_k^z}|\,d\mu .
\end{equation}
Applying the $s$-Poincar\'e inequality to $B_k$, and observing that
\[
\lambda B_k^z\subset \lambda B_0^z\subset K,
\]
we obtain
\begin{align*}
|u_{B_{k+1}^z}-u_{B_k^z}|
&\lesssim 2^{-k}r
\left(
\fint_{\lambda B_k^z}
(\chi_K\operatorname{Lip}u)^s\,d\mu
\right)^{1/s}
\\
&\lesssim
2^{-k}r\,M_s\bigl(\chi_K\operatorname{Lip}u\bigr)(z).
\end{align*}
Therefore,
$$
\begin{aligned}
|u(z)-u_{B_0^z}|
&\le
\sum_{k=0}^{\infty}
|u_{B_{k+1}^z}-u_{B_k^z}|
\lesssim r\,M_s\bigl(\chi_K\operatorname{Lip}u\bigr)(z).
\end{aligned}
$$
This finishes the proof of \eqref{eq:point-average-estimate}.

Let $\delta:=\dist(U,X\setminus K)>0$ and set
$$
\rho
:=
\min\left\{
\diam U,\frac{\delta}{16\lambda}
\right\}.
$$
Notice that 
\begin{equation}\label{eq:ball-in-K}
B(z,r)\subset K\quad\mbox{whenever $z\in U$ and $0<r<\delta$.}
\end{equation}
Indeed, otherwise there would exist $w\in X\setminus K$ with
$d(z,w)<\delta$, contradicting the definition of $\delta$.

To proceed, we divide the proof into two cases.
\medskip

\noindent
\textbf{Case 1: $d(x,y)<\rho$}. Define 
$$
B_x:=B(x,d(x,y))\quad\mbox{and}\quad B_y:=B(y,d(x,y)).
$$
Observe that $B_x\cup B_y\subset 2B_x$. 
Notice that if $z\in U$ and $0<r<\delta$, then $B(z,r)\subset K$. Thus, since $2\lambda d(x,y)\leq 2\lambda\rho<\delta$,
we have that $\lambda 2B_x\subset K$. Granted this, we can apply \eqref{eq:point-average-estimate} to the balls $B_x$ and $B_y$ to conclude that
\begin{equation}\label{eq:point-average-estimate-X}
|u(x)-u_{B_x}|
\leq
Cd(x,y)\,M_s\bigl(\chi_K\operatorname{Lip}u\bigr)(x),
\end{equation}
and
\begin{equation}\label{eq:point-average-estimate-Y}
|u(y)-u_{B_y}|
\leq
Cd(x,y)\,M_s\bigl(\chi_K\operatorname{Lip}u\bigr)(y).
\end{equation}
We now compare the two averages over $B_x$ and $B_y$. Since 
$B_x\subset 2B_x$ and $B_y\subset 2B_x$ with $\mu(B_x)\approx\mu(B_y)\approx\mu(2B_x)$, we have
\begin{equation}
\label{eq: middle estimate}
\begin{aligned}
|u_{B_x}-u_{B_y}|
&\leq
|u_{B_x}-u_{2B_x}|
+
|u_{B_y}-u_{2B_x}|
\\
&\lesssim 
\fint_{2B_x}|u-u_{2B_x}|\,d\mu
\\
&\lesssim d(x,y)
\left(
\fint_{\lambda 2B_x}
(\chi_K\operatorname{Lip}u)^s\,d\mu
\right)^{1/s}
\\
&\lesssim d(x,y)\,M_s\bigl(\chi_K\operatorname{Lip}u\bigr)(x).
\end{aligned}
\end{equation}
Combining this with \eqref{eq:point-average-estimate-X} and \eqref{eq:point-average-estimate-Y} implies that the inequality \eqref{eq: Banach maximal fun upper grad in nbhd} holds in Case 1.
\medskip

\noindent
\textbf{Case 2: $d(x,y)\geq\rho$.} In this case, we start by choosing a maximally $\rho$-separated set $\{z_i\}_{i\in I}\subset U$. Observe that
$$
U\subset\bigcup_{i\in I}B(z_i,\rho).
$$
Since the balls $B(z_i,\rho/2)$ are pairwise disjoint, and $U$ is bounded, it follows from the Ahlfors regularity that the index set $I$ is finite, where the bound on its cardinality depends only on $\mu$, $\rho$, and $U$.

We claim that we can connect $x$ and $y$ by a finite chain of these $\rho$-balls. More specifically, choose indices $i_x,i_y\in I$ such that
\[
x\in B(z_{i_x},\rho)
\qquad\text{and}\qquad
y\in B(z_{i_y},\rho).
\]
We will show that there exists a finite sequence of indices
$i_x=j_1,j_2,\ldots,j_m=i_y$
such that
\begin{equation}
\label{chainballs}
B(z_{j_k},\rho)\cap B(z_{j_{k+1}},\rho)\neq\varnothing
\end{equation}
for every $k=1,\ldots,m-1$. To prove this claim, for each $j\in I$, set
\[
V_j:=B(z_j,\rho)\cap U.
\]
Then each $V_j$ is a nonempty relatively open subset of $U$, and
\[
U=\bigcup_{j\in I} V_j.
\]
Let $I_x$ be the set of all $j\in I$ having the following property: there exist an integer $\ell\geq1$ and indices $i_x=j_1,j_2,\ldots,j_\ell=j$ such that 
\[
V_{j_k}\cap V_{j_{k+1}}\neq\varnothing
\]
for every $k=1,\ldots,\ell-1$. Note that $I_x\neq\varnothing$ since $i_x\in I_x$. Let 
\[
U_1:=\bigcup_{j\in I_x}V_j
\qquad\text{and}\qquad
U_2:=\bigcup_{j\in I\setminus I_x}V_j
\]
Then $U_1$ and $U_2$ are relatively open in $U$ and $U=U_1\cup U_2$. Moreover, $U_1$ and $U_2$  are disjoint. Indeed, suppose that $U_1\cap U_2\neq\varnothing$. Then there exist $j\in I_x$ and $k\in I\setminus I_x$ such that $V_j\cap V_k\neq\varnothing$.
Since $j\in I_x$, there are indices $i_x=j_1,j_2,\ldots,j_\ell=j$
such that $V_{j_k}\cap V_{j_{k+1}}\neq\varnothing$
for every $k=1,\ldots,\ell-1$. Thus, $V_{i_x}, V_{j_2},\dots V_{j}$
is a chain connecting $V_{i_x}$ and $V_{j}$. Since $V_j\cap V_k\neq\varnothing$, we have that the sets $V_{i_x}, V_{j_2},\dots V_{j}, V_k$ form
a chain connecting $V_{i_x}$ and $V_{k}$. Hence, $k\in I_x$, which is a contradiction. Thus,  $U_1\cap U_2=\varnothing$, as wanted. Combining these observations with the fact that $i_x\in I_x$ implies $U_1\neq\varnothing$, we can deduce from the connectedness of $U$ that $U_2=\emptyset$. Hence, $I_x=I$. In particular, $i_y\in I=I_x$ and therefore,  $x$ and $y$ can be connected by a finite chain of these $\rho$-balls. 

Note that we can assume the indices in \eqref{chainballs} are distinct.
For each $j\in I$, set $B_{j}:=B(z_{j},\rho)$. Then we can write:
\begin{equation}
\label{chainingestimate}
\begin{aligned}
|u(x)-u(y)|&\leq
|u(x)-u_{B_{i_x}}|
+
\sum_{k=1}^{m-1}
|u_{B_{j_k}}-u_{B_{j_{k+1}}}|
+
|u_{B_{i_y}}-u(y)|
\\
&=:J_1+J_2+J_3.
\end{aligned}
\end{equation}

For $J_1$, we write
\begin{equation}\label{eq:J1-1}
J_1\leq|u(x)-u_{B(x,2\rho)}|+|u_{B(x,2\rho)}-u_{B_{i_x}}|
\end{equation}
Since $2\lambda\rho<\delta$, by \eqref{eq:ball-in-K}
we have that $B(x,2\lambda \rho)\subset K$. Granted this, we can apply \eqref{eq:point-average-estimate} to the ball $B(x,2\rho)$ to conclude that 
\begin{equation}\label{eq:J1-2}
|u(x)-u_{B(x,2\rho)}|\lesssim \rho\, M_s\bigl(\chi_K\operatorname{Lip}u\bigr)(x).
\end{equation}
Moreover, since $x\in B_{i_x}$, we have that $B_{i_x}\subset B(x,2\rho)\subset 3B_{i_x}$. Hence, $\mu(B_{i_x})\approx\mu(B(x,2\rho))$ by Ahlfors regularity. Granted this, we can apply the $s$-Poincar\'e inequality to $B(x,2\rho)$, where $B(x,2\lambda \rho)\subset K$, to deduce that
\[
\begin{aligned}
|u_{B(x,2\rho)}-u_{B_{i_x}}|&\leq
\fint_{B_{i_x}}|u-u_{B(x,2\rho)}|\,d\mu
\\
&\lesssim\fint_{B(x,2\rho)}|u-u_{B(x,2\rho)}|\,d\mu
\\
&\lesssim 2\rho
\left(
\fint_{B(x,2\lambda\rho)}
(\chi_K\operatorname{Lip}u)^s\,d\mu
\right)^{1/s}
\\
&\lesssim\rho\,
M_s\bigl(\chi_K\operatorname{Lip}u\bigr)(x).
\end{aligned}
\]
Combining this with \eqref{eq:J1-1} and \eqref{eq:J1-2}
gives
\begin{equation}\label{eq:J1-3}
J_1\lesssim \rho\,M_s\bigl(\chi_K\operatorname{Lip}u\bigr)(x).
\end{equation}
Similarly, we obtain
\begin{equation}\label{eq:J3-1}
J_3\lesssim \rho\,M_s\bigl(\chi_K\operatorname{Lip}u\bigr)(y).
\end{equation}
To estimate $J_2$, fix $k\in\{1,\dots,m-1\}$. By construction, $B_{j_k}\cap B_{j_{k+1}}\neq\varnothing$, from which we can conclude that $B_{j_k}\subset 3B_{j_{k+1}}\subset 5B_{j_k}$, where $3\lambda B_{j_{k+1}}\subset K$ by \eqref{eq:ball-in-K}. Granted this, we can appeal to Ahlfors regularity and the $s$-Poincar\'e inequality, to conclude that
\begin{equation}
\label{eq:J_2}
\begin{aligned}
|u_{B_{j_k}}-u_{B_{j_{k+1}}}|
&\leq
|u_{B_{j_k}}-u_{3B_{j_{k+1}}}|+|u_{3B_{j_{k+1}}}-u_{B_{j_{k+1}}}|
\\
&\leq\fint_{B_{j_k}}|u-u_{3B_{j_{k+1}}}|\,d\mu
+\fint_{B_{j_{k+1}}}|u-u_{3B_{j_{k+1}}}|\,d\mu
\\
&\lesssim
\fint_{3B_{j_{k+1}}}|u-u_{3B_{j_{k+1}}}|\,d\mu
\\
&\lesssim 3\rho
\left(
\fint_{3\lambda B_{j_{k+1}}}
(\chi_K\operatorname{Lip}u)^s\,d\mu
\right)^{1/s}.
\end{aligned}
\end{equation}
To estimate the last integral average, observe that
since $3\lambda\rho+\diam U\leq 4\lambda\diam U$, we have
$$
3\lambda B_{j_{k+1}}=B(z_{j_{k+1}},3\lambda\rho)\subset B(x,4\lambda\diam U)\subset 
B(z_{j_{k+1}},5\lambda\diam U).
$$
Thus, by Ahlfors regularity we have
\begin{equation*}
\begin{aligned}
\left(
\fint_{3\lambda B_{j_{k+1}}}
(\chi_K\operatorname{Lip}u)^s\,d\mu
\right)^{1/s}
&\lesssim\left(
\fint_{B(x,4\lambda\diam U)}
(\chi_K\operatorname{Lip}u)^s\,d\mu
\right)^{1/s}
\\
&\leq
M_s\bigl(\chi_K\operatorname{Lip}u\bigr)(x).
\end{aligned}
\end{equation*}
All together, this, \eqref{eq:J_2}, the fact that $m\leq\#I<\infty$, \eqref{chainingestimate}, \eqref{eq:J1-3}, and \eqref {eq:J3-1}, 
and imply that the inequality \eqref{eq: Banach maximal fun upper grad in nbhd} also holds in Case 2.
Hence, $h_u=C\,M_s(\chi_K\operatorname{Lip}u)$ is a Haj\l asz  gradient of $u$ on $U$, for some $C=C(\lesssim)$ determined above.
\end{proof}

We also record the following, which is a counterpart of Lemma~\ref{lem: Haj grad for Lip} without the connectedness assumption, but with a restriction on the allowed distance to the boundary for the ``internal'' subset $U$ of $K$.

\begin{lemma}\label{lem: Haj grad for Lip - OLD}
Let $U\subset X$ be a nonempty measurable set, and let $K\subset X$ be compact with
$$
\lambda\diam U<\dist(U,X\setminus K), 
$$
where $\lambda$ is the Poincar\'e radial constant.
Suppose $u:X\to V$ is Lipschitz on a neighborhood of $K$, and let $s\geq1$ be as in the Poincar\'e inequality \eqref{eq: PI}. Then there is a constant $C=C(\mu,PI, U, K)>0$, such that
$$
h_u:=C\,M_s\bigl(\chi_K\operatorname{Lip}u\bigr)
$$
is a Haj\l asz  gradient of $u$ on $U$, i.e.,
\begin{equation}\label{eq: Banach maximal fun upper grad in nbhd - OLD}
    \|u(x)-u(y)\|
\le
d(x,y)\bigl(h_u(x)+h_u(y)\bigr)
\end{equation}
for a.e. $x,y\in U$.
\end{lemma}

\begin{proof}
As in the proof of Lemma~\ref{lem: Haj grad for Lip}, we can assume $u$ is a globally Lipschitz real-valued function, and that the points $x,y\in U$ are such that the maximal functions below are finite. 

Since $\lambda\diam U<\dist(U,X\setminus K)$, we can choose $\varepsilon\in(0,1)$ small enough so that
\begin{equation}
\label{eq: epsilon fat}   
(1+\varepsilon)\lambda\diam U<\dist(U,X\setminus K)
\end{equation}
Let $B_x:=B(x, \varepsilon d(x,y))$, $B_y:=B(y,\varepsilon d(x,y))$,  and $B^*:=B(x,(1+\varepsilon)d(x,y))$, and write 
\begin{equation}\label{eq:3.3 neq}
|u(x)-u(y)|\leq|u(x)-u_{B_x}|+|u_{B_x}-u_{B_y}|+|u_{B_y}-u(y)|.
\end{equation}
Observe that by \eqref{eq: epsilon fat},
$$
\lambda B_x\cup \lambda B_y\subset \lambda B^*
\subset K,
$$
and so the point-to-average estimate \eqref{eq:point-average-estimate}
gives for $z\in\{x,y\}$,
\begin{equation}\label{eq:ptavgest2}
|u(z)-u_{B_z}|
\leq
Cd(x,y)\,M_s\bigl(\chi_K\operatorname{Lip}u\bigr)(z),
\end{equation}
On the other hand, arguing as in the proof of \eqref{eq: middle estimate} gives 
$$
\begin{aligned}
\left|u_{B_x}-u_{B_y}\right|
&\leq
\left|u_{B_x}-u_{B^*}\right|
+
\left|u_{B_y}-u_{B^*}\right| \\
&\lesssim_{\varepsilon}
\fint_{B^*}
\left|u-u_{B^*}\right|\,d\mu \\
&\lesssim_{\varepsilon}
d(x,y)
\left(
\fint_{\lambda B^*}
\bigl(\chi_K \operatorname{Lip}u\bigr)^s\,d\mu
\right)^{1/s} \\
&\lesssim_{\varepsilon}
d(x,y)\,
M_s\bigl(\chi_K \operatorname{Lip}u\bigr)(x).
\end{aligned}
$$
This, along with \eqref{eq:ptavgest2} and \eqref{eq:3.3 neq}, implies $h_u:=C\,M_s(\chi_K\operatorname{Lip}u)$ is a Haj\l asz  gradient of $u$ on $U$, for some $C$ determined above. Note that this constant $C$ depends on $\varepsilon$ and deteriorates as $\varepsilon\to0$.
\end{proof}

Lastly, we record a convergence lemma that allows us to restrict our attention to Lipschitz sequences instead.

\begin{lemma}\label{le: limit of Haj is Haj}
Let $U\subset X$ with $\mu(U)<\infty$. Suppose $u_j\to u$ in $L^p(U\colon V)$, and each $u_j$ has a Haj\l asz  gradient $h_j\in L^p(U)$ satisfying
$$
\sup_j \|h_j\|_{L^p(U)}<\infty.
$$
Then $u\in M^{1,p}(U\colon V)$.
\end{lemma}

\begin{proof}
Since $p>1$, the space $L^p(U)$ is reflexive. Passing to a subsequence, we may assume that
$$
h_j\rightharpoonup h
\qquad\text{weakly in }L^p(U)
$$
for some $h\in L^p(U)$. By Mazur's lemma (see for instance \cite[p.~19]{Juha-Jeremy-etc-book}), there are convex combinations 
$$
g_N=\sum_{j\ge N} a_{j,N}h_j,
$$
of $h_j$, where each sum is finite, $a_{j,N}\ge 0$, and $\sum_{j\ge N}a_{j,N}=1$, with $g_N$ converging strongly to $h$ in $L^p(U)$ as $N\rightarrow\infty$.

Define the convex combination of the corresponding $u_j$ functions by
$$
v_N=\sum_{j\ge N} a_{j,N}u_j.
$$
Since $u_j\to u$ in $L^p(U \colon V)$, we also have $v_N\to u$ in $L^p(U \colon V)$.
Moreover, $g_N$ is a Haj\l asz  gradient of $v_N$, because
\begin{equation}\label{eq: convex comb est}
\begin{aligned}
\|v_N(x)-v_N(y)\|
&\le
\sum_{j\ge N} a_{j,N} \|u_j(x)-u_j(y)\|
\\
&\le
d(x,y)
\sum_{j\ge N} a_{j,N}(h_j(x)+h_j(y))
\\
&=
d(x,y)(g_N(x)+g_N(y))
\end{aligned}
\end{equation}
for a.e. $x,y\in U$.

Passing to further subsequences if necessary, we may assume
$$
v_N(x)\to u(x),
\qquad
g_N(x)\to h(x)
$$
for a.e. $x\in U$. 
Taking limits along these subsequences in \eqref{eq: convex comb est} gives
$$
\|u(x)-u(y)\|
\le
d(x,y)(h(x)+h(y)),
$$
for a.e. $x,y\in U$. Thus, $h\in L^p(U)$ is a Haj\l asz  gradient of $u$, and hence $u\in M^{1,p}(U\colon V)$.
\end{proof}

We are now ready to establish the admittedly more challenging inclusion from Theorem~\ref{thmA:almost-characterization}, treated separately from the others for the sake of the exposition.

\begin{theorem}\label{thm: CH are M1ploc}
If $f:X\to V$ is $(p,\alpha_p)$-compactly H\"older, then $f\in M^{1,p}_{\text{loc}}(X\colon V).
$
\end{theorem}

\begin{proof}
Let $B_0=B(x_0,R)$. Since $X$ is complete and doubling, it is proper. Hence, closed and bounded sets are compact. Set
$$
K=\overline{B(x_0,(2\lambda+2)R)},
$$ which guarantees that
$$
\dist(B_0,X\setminus K)\geq(2\lambda+1)R>2\lambda R\geq\lambda\diam B_0,
$$
since $\diam B_0\le 2R$.

By Lemmas \ref{le: partition of unity} and \ref{le: compHold gives Lip in Lp}, there are $t_0>0$ and $C_K<\infty$ such that
\begin{equation}\label{eqThm: Lip in Lp}
    \int_K \bigl(\operatorname{Lip}\Phi_t f\bigr)^p\,d\mu
\leq C_K,
\end{equation} for all $0<t<t_0$.
Fix a sequence $t_j\searrow 0$ with
$$
t_j<\min\{ t_0, 1/3, R, \dist(B_0,X\setminus K)/3 \},
$$
and set
$$
u_j:=\Phi_{t_j}f.
$$

Since $f$ is continuous on the compact set $K$, it is uniformly continuous on $K$. We claim that
$$
u_j\to f
\qquad\text{uniformly in }B_0.
$$
Indeed, if $x\in B_0$ and $\phi_i(x)\neq 0$, then $x_i\in B(x,2t_j)$, and hence
$$
B(x_i,t_j)\subset B(x,3t_j)\subset K.
$$
Therefore, by choice of $\Phi_t f$ and $\sum_i \phi_i=1$, we have
$$
\begin{aligned}
\|u_j(x)-f(x)\|
&=
\left\|
\sum_i \phi_i(x)\bigl(f_{B(x_i,t_j)}-f(x)\bigr)
\right\|
\\
&\le
\sum_i\phi_i(x)
\fint_{B(x_i,t_j)}
\|f(y)-f(x)\|\,d\mu(y)
\\
&\le
\sup_{y\in B(x,3t_j)} \|f(y)-f(x)\|.
\end{aligned}
$$
The last quantity tends to $0$ uniformly for $x\in B_0$, by uniform continuity of $f$ on $K$. Hence $u_j\to f$ uniformly on $B_0$.

Since $B_0$ has finite measure, uniform convergence implies $L^p$-convergence, i.e.,
$$
\|u_j-f\|_{L^p(B_0)}
\le
\mu(B_0)^{1/p}\sup_{x\in B_0}\|u_j(x)-f(x)\|
\longrightarrow 0.
$$
By Lemma \ref{lem: Haj grad for Lip - OLD}, the function $h_j:=C\,M_s(\chi_K\operatorname{Lip}u_j)$ is a Haj\l asz  gradient of $u_j$ on $B_0$. Since $p>s$, the Maximal Function Theorem and \eqref{eqThm: Lip in Lp} give
$$
\|h_j\|_{L^p(B_0)}
\le
\|h_j\|_{L^p(X)}
\lesssim
\|\operatorname{Lip}u_j\|_{L^p(K)}
\le C_K^{1/p}.
$$
Thus the approximating functions $u_j$ converge to $f$ in $L^p(B_0 \colon V)$, and they admit Haj\l asz  gradients uniformly bounded in $L^p(B_0\colon V)$. By Lemma \ref{le: limit of Haj is Haj} this implies that $f\in M^{1,p}(B_0\colon V)$.
Since $B_0$ was arbitrary, we have that $f\in M^{1,p}_{\text{loc}}(X\colon V)$, completing the proof.
\end{proof}

The proof of the main characterization result now follows.

\begin{proof}[Proof of Theorem~\ref{thmA:almost-characterization}]
    The inclusion $CH_m^{p,\alpha_p}(X \colon V)\subset CH^{p,\alpha_p}(X \colon V)$ follows directly by Definition~\ref{def: CH maps}, and the inclusion $CH^{p,\alpha_p}(X\colon V)\subset M^{1,p}_\loc(X \colon V)$ follows by Theorem~\ref{thm: CH are M1ploc}. 
    The fact that $M^{1,p}_{\loc}(X \colon V)=N_\loc^{1,p}(X \colon V)$ can be found in \cite[Theorem~10.5.3]{Juha-Jeremy-etc-book}, and the inclusion $N_\loc^{1,p}(X \colon V)\subset \bigcap_{Q<q<p} CH^{q,\alpha_q}_m(X \colon V)$ can be found in \cite{AC-G26}.
    
    We focus on the remaining inclusion $M^{1,p}_\loc(X \colon V)\subset CH^{p,\alpha_p}(X \colon V)$. Given arbitrary $\eps\in (0,1)$ and $E\subset X$ compact, fix a tiny neighborhood of $E$, say 
    $$
    \Omega=\bigcup_{z\in E}B(z,1/10).
    $$ Since $X$ is proper, we have that $K:=\overline{\Omega}$ is compact. Fix $r_M>0$ be the upper threshold for the radii in
Lemma~\ref{DOUBembedding-cpt}, applied to the compact set $K$.
Choose $0<r_E<r_M$ small enough so that $B(z,10r_E)\subset K$ for every $z\in E$. 
When $X$ is bounded, also require that
$(2+\eps)r_E<\diam X.$
Suppose $\{B_i:=B(x_i,r)\}_{i\in I}$ is as in Definition~\ref{def: CH maps} with $r<r_E$. Then by Lemma \ref{DOUBembedding-cpt} there is some $g\in D(f_K)\cap L^p(K)$ so that
    \begin{equation}\label{eq: MS holder norm bound}
        |f|_{\alpha_p, B_i}^p\lesssim \int_{2B_i} g(z)^pd\mu(z).
    \end{equation} For $z\in \cup_{i\in I}2B_i$ let
    $$
    I(z):=\{i\in I: z\in 2B_i\}.
    $$ Then by Ahlfors regularity
    $$
    \# I(z) C_\mu^{-1} (\eps r)^Q\leq \sum_{i\in I(z)}\mu (B(x_i, \eps r))=\mu\left(\bigcup_{i\in I(z)} B(x_i, \eps r)\right).
    $$ But for any $i\in I(z)$ we have $\eps B_i\subset B(z,2r+\eps r)$. So the above and Ahlfors regularity give
    $$
    \# I(z)\leq\frac{C_\mu^2 (2r+\eps r)^Q}{\eps^Qr^Q}\leq \frac{C_\mu^2 (2+\eps )^Q}{\eps^Q}.
    $$ Summing over all $i\in I$ in \eqref{eq: MS holder norm bound} and using the above we have
    $$
    \sum_{i\in I}|f|_{\alpha_p, B_i}^p\lesssim \int_K g(z)^p \sum_{i\in I} \chi_{2B_i}(z)d\mu(z)\leq \int_K g(z)^p d\mu(z) \,\frac{C_\mu^2 (2+\eps )^Q}{\eps^Q},
    $$ which proves the compactly H\"older condition for the arbitrary $f\in M^{1,p}_\loc(X \colon V)$.
\end{proof}

\begin{remark}
    Note that in \cite{Chron_comp_holder_Minkowski} it is proved that $N^{1,p}_\loc(X \colon V)\subset CH^{q,\alpha_q}(X \colon V)$ for all $q\in (Q,p)$, with a similar integrability drop to that of the modified compactly H\"older case. However, the reason that the end point case $q=p$ is not established in \cite{Chron_comp_holder_Minkowski} is that the underlying space $X$ is assumed to be locally $Q$-homogeneous, which is a generally weaker property than $Q$-Ahlfors regularity. Nonetheless, this was still enough to establish the desired dimension distortion results in that paper. Similarly, while the dimension distortion results in \cite{AC-G26} were established for the modified compactly H\"older class under the assumption of local $Q$-homogeneity for $X$, these results are true for the generally bigger $CH^{p,\alpha_p}(X \colon V)$ class if $X$ admits the stronger conditions in Theorem~\ref{thmA:almost-characterization}.
\end{remark}

\section{Compactly H\"older and Haj\l asz seminorms}

Throughout this section we assume $(X,d,\mu)$ is a complete $Q$-Ahlfors regular metric measure space supporting a $p$-Poincar\'e inequality  with $p>Q\geq1$. 

Let $K\subset X$ be compact, let $p>Q$, and set
$$
    \alpha_p:=1-\frac Qp.
$$
For $\eps\in(0,1)$ and $0<r<1$, let
$\mathcal B_{\eps,K}(r)$ be the collection of all at most countable
families $\mathcal B=\{B(x_i,\rho)\}_{i\in I}$  of balls $B_i=B(x_i,\rho)$, such that
$$
    K\subset \bigcup_{i\in I}B_i,\qquad
    B_i\cap K\neq\emptyset,\qquad
    0<\rho<r,
$$
and
$$
    B(x_i,\eps \rho)\cap B(x_j,\eps \rho)=\emptyset
    \qquad\text{for } i\neq j.
$$
Define
$$
    [f]_{p,\alpha_p,\eps}(K)^p
    :=
    \inf_{0<r<1}
    \sup_{\mathcal B\in\mathcal B_{\eps,K}(r)}
    \sum_{B_i\in\mathcal B}|f|_{\alpha_p,B_i}^p.
$$
Equivalently, $[f]_{p,\alpha_p,\eps}(K)^p$ is the infimum of such constants
$C$ such that there exists $r_K\in(0,1)$ with the property that
$$
    \sum_{B_i\in\mathcal B}|f|_{\alpha_p,B_i}^p\le C
$$
for every $\mathcal B\in\mathcal B_{\eps,K}(r_K)$.

We are now ready to prove the comparability of the corresponding seminorms.

\begin{proof}[Proof of Theorem~\ref{thmB:seminorm-comparison}]
Let $f:X\to V$ be continuous. We first prove the estimate
$$
    [f]_{p,\alpha_p,\eps}(K)
    \lesssim_{\eps,p,K,\Omega}
    \|f\|_{\dot{M}^{1,p}(\Omega)}.
$$
If the right-hand side is infinite, there is nothing to prove. Otherwise, let $g\in L^p(\Omega)$ be a Haj{\l}asz gradient of $f$ on $\Omega$. Set
$$
K_1:=\{x\in X:\dist(x,K)\le 1\}.
$$
Since $X$ is proper and $K$ is compact, $K_1$ is compact.
Let $r_{M}\in(0,1/3)$ be any upper threshold for the radii satisfying the Morrey estimate in Lemma~\ref{DOUBembedding-cpt}  for the compact set $K_1$. Since $K\subset\Omega$ and $K$ is compact, we can choose $r_0\in(0,r_{M})$
sufficiently small that
\begin{equation}\label{eq: seminorm Lambda r_0}
    3r_0
    <
    \dist(K,X\setminus\Omega).
\end{equation}
If $X$ is bounded then we also require that
$(2+\eps)r_0<\diam X.$

Given $\rho\in(0,r_0)$, let
$$
    \mathcal B=\{B(x_i,\rho)\}_{i\in I}
    \in \mathcal B_{\eps,K}(r_0),
$$ and set $B_i=B(x_i,\rho)$ for all $i$.
Due to $B_i\cap K\neq\emptyset$ and \eqref{eq: seminorm Lambda r_0} we have $x_i\in K_1$ and $2B_i\subset\Omega$ for every $i$. Thus, applying Lemma~\ref{DOUBembedding-cpt} (the Morrey estimate) yields
\begin{equation}\label{eq: ch norm upper big ball}
    |f|_{\alpha_p,B_i}^p
    \lesssim
    \int_{2 B_i} g^p\,d\mu.
\end{equation}
For each $z\in \cup_{i\in I}2B_i$, let
$$
I(z):=\{i\in I: z\in 2 B_i\}.
$$ 
Then $z\in K_1$ and by Ahlfors regularity
$$
    \# I(z) C_\mu^{-1} (\eps \rho)^Q\leq \sum_{i\in I(z)}\mu (B(x_i, \eps \rho))=\mu\left(\bigcup_{i\in I(z)} B(x_i, \eps \rho)\right).
    $$ But for any $i\in I(z)$ we have $\eps B_i\subset B(z,(2+\eps)\rho)$. So the above and Ahlfors regularity gives
    $$
    \# I(z)\leq\frac{C_\mu^2 (2+\eps)^Q\rho^Q}{\eps^Q\rho^Q}\leq \frac{C_\mu^2 (2+\eps )^Q}{\eps^Q}.
    $$ 
    Summing over all $i\in I$ in \eqref{eq: ch norm upper big ball} and using the above we have
    $$
    \sum_{i\in I}|f|_{\alpha_p, B_i}^p\lesssim \int_\Omega g^p \sum_{i\in I} \chi_{2 B_i}d\mu\leq \int_\Omega g^p d\mu \,\frac{C_\mu^2 (2 +\eps )^Q}{\eps^Q}.
    $$
Taking the supremum over all
$\mathcal B\in\mathcal B_{\eps,K}(r_0)$ and then using the
definition of $[f]_{p,\alpha_p,\eps}(K)$ gives
$$
[f]_{p,\alpha_p,\eps}(K)^p
\lesssim
\|g\|_{L^p(\Omega)}^p.
$$
Finally, taking the infimum over all Haj{\l}asz gradients $g$ of
$f$ on $\Omega$ yields
$$
[f]_{p,\alpha_p,\eps}(K)
\lesssim
\|f\|_{\dot M^{1,p}(\Omega)}.
$$

We now prove the left-hand-side estimate
$$
    \|f\|_{\dot{M}^{1,p}(U)}
    \lesssim
    [f]_{p,\alpha_p,\eps}(K).
$$
If $[f]_{p,\alpha_p,\eps}(K)^p=\infty$, the relation is trivial, and so assume that $[f]_{p,\alpha_p,\eps}(K)^p<\infty$. 
Let
$\delta>0$. By the definition of the semi-norm, there exists $r_K=r_K(\delta)\in(0,1)$ such
that
\begin{equation}\label{eq: semi-norm pf CH-best-constant}
    \sum_{B_i\in\mathcal B}|f|_{\alpha_p,B_i}^p
    \le [f]_{p,\alpha_p,\eps}(K)^p+\delta
\end{equation}
for every $\mathcal B\in\mathcal B_{\eps,K}(r_K)$.

Let $A\ge1$ be the constant from Lemma~\ref{le: partition of unity}. Since $0<\eps<1/8$, we have
\begin{equation}\label{eq: sigma-choice}
    4\eps A\le A.
\end{equation}
Choose $t>0$  so small that
\begin{equation}\label{eq: seminorm Asigma r_K}
    2At<\min\{r_K,\diam X\}.
\end{equation}
By Lemma~\ref{le: partition of unity}, there exists a locally Lipschitz mapping $\Phi_tf:X\rightarrow V$ satisfying
$$
    \operatorname{Lip}(\Phi_t f)(x)
    \lesssim
    \frac1t
    \fint_{B(x,At)}
    \|f-f_{B(x,At)}\|\,d\mu,
    \qquad x\in K.
$$
Arguing as in the proof of Lemma \ref{le: compHold gives Lip in Lp} and using $\alpha_p=1-Q/p$ yields
\begin{equation}\label{eq: Lip-Phit}
    \bigl(\operatorname{Lip}(\Phi_t f)(x)\bigr)^p
    \lesssim
    t^{-Q}|f|_{\alpha_p,B(x,At)}^p,
    \qquad x\in K.
\end{equation}

Let $\{z_i\}_{i=1}^N\subset K$, be a maximal $A t$-separated set. Then $N\in\mathbb{N}$ and
$$
    K\subset \bigcup_{i=1}^N B(z_i,A t).
$$
Choose a measurable partition $\{E_i\}_{i=1}^N$ of $K$ such that
$$
    E_i\subset B(z_i,A t).
$$
If $x\in E_i$, then $B(x,At)\subset B(z_i,2At)$.
Thus, by \eqref{eq: Lip-Phit} and Ahlfors regularity,
\begin{equation}\label{eq: seminorm Lip Lq norm}
\begin{aligned}
    \int_K(\operatorname{Lip}\Phi_t f)^p\,d\mu
    &\lesssim
    \sum_{i=1}^N
    \int_{E_i}
    t^{-Q}|f|_{\alpha_p,B(z_i,2At)}^p\,d\mu
    \\
    &\lesssim
    \sum_{i=1}^N
    |f|_{\alpha_p,B(z_i,2At)}^p .
\end{aligned}
\end{equation}
Set $B_i:=B(z_i,2At)$.
Then $\{B_i\}_{i=1}^N$ covers $K$, every $B_i$ meets $K$, and by \eqref{eq: seminorm Asigma r_K} we have $\operatorname{rad}(B_i)<r_K$.
Moreover, by \eqref{eq: sigma-choice} we have
$$
    2\eps\operatorname{rad}(B_i)
    =
    4\eps At
    \le A t.
$$
Since the centers $z_i$ are $A t$-separated, the balls $B(z_i,2\eps At)$ are pairwise disjoint. Hence, $\{B_i\}_{i=1}^N\in\mathcal B_{\eps,K}(r_K)$. As a result, by \eqref{eq: semi-norm pf CH-best-constant} and \eqref{eq: seminorm Lip Lq norm} we have
\begin{equation}\label{eq: seminorm Lip-Phit-bound}
    \int_K(\operatorname{Lip}\Phi_t f)^p\,d\mu
    \lesssim
    [f]_{p,\alpha_p,\eps}(K)^p+\delta .
\end{equation}

Choose a sequence $t_j\to0$ satisfying 
$$2At_j<\min\{r_K,\diam X\}$$ 
and $3t_j<\dist(U,X\setminus K)$, and set
$$
    u_j:=\Phi_{t_j}f.
$$
Then, if $x\in U$ and $\phi_i(x)\neq 0$, it follows from definition of $\Phi_t$ that $x_i\in B(x,2t_j)$, and hence
$$
B(x_i,t_j)\subset B(x,3t_j)\subset K.
$$
Consequently,
\[
    \|u_j(x)-f(x)\|
    \le
    \sup_{y\in B(x,3t_j)}\|f(y)-f(x)\|.
\]
Since $f$ is uniformly continuous on the compact set $K$,
we conclude that $u_j\to f$ uniformly on $U$, and hence in
$L^p(U\colon V)$.
Applying Lemma~\ref{lem: Haj grad for Lip} yields a constant $C>0$, that is independent of $j$ and has the property that
$$
    h_j:=C\,M_s(\chi_K\operatorname{Lip}u_j)
$$
is a Haj\l asz  gradient of $u_j$ on $U$. 
As in the proof of
Theorem~\ref{thm: CH are M1ploc}, since $p>s$, the
Hardy--Littlewood maximal theorem and
\eqref{eq: seminorm Lip-Phit-bound} give
$$
\|h_j\|_{L^p(U)}
\lesssim
\|\operatorname{Lip}u_j\|_{L^p(K)}
\lesssim
\left(
[f]_{p,\alpha_p,\eps}(K)^p+\delta
\right)^{1/p}.
$$
Thus, since $u_j\to f$ in $L^p(U\colon V)$, 
Lemma~\ref{le: limit of Haj is Haj} implies that
$f\in M^{1,p}(U\colon V)$.
Moreover, the proof of that lemma and the weak
lower semicontinuity of the $L^p$ norm give
\[
\begin{aligned}
\|f\|_{\dot M^{1,p}(U)}
&\le
\liminf_{j\to\infty}\|h_j\|_{L^p(U)}\\
&\lesssim
\bigl(
[f]_{p,\alpha_p,\eps}(K)^p+\delta
\bigr)^{1/p}.
\end{aligned}
\]
Since the implicit constant is independent of $\delta$, letting
$\delta\to0$ gives
\[
\|f\|_{\dot M^{1,p}(U)}
\lesssim
[f]_{p,\alpha_p,\eps}(K).
\]
This completes the proof.
\end{proof}

\begin{corollary}\label{cor: ball-CH-Haj-comparison}
Let $B:=B(x,r)$ be a ball and $0<\eps<1/8$. For all $\theta\in(1,\infty)$ for which $\theta^{-1}B$ is connected, there exists a positive constant $C$ such that
$$
    C^{-1}\|f\|_{\dot{M}^{1,p}(\theta^{-1}B)}
    \leq
    [f]_{p,\alpha_p,\eps}(\overline B)
    \leq C
    \|f\|_{\dot{M}^{1,p}(\theta B)},
$$
for every continuous $f:X\to V$.
\end{corollary}

Arguing as in the proof of Theorem~\ref{thmB:seminorm-comparison}, with Lemma~\ref{lem: Haj grad for Lip - OLD} in place of Lemma~\ref{lem: Haj grad for Lip} gives a version of Theorem~\ref{thmB:seminorm-comparison}, where the set $U$ is not necessarily connected; we omit the details.
\begin{theorem}\label{thm: CH-Haj seminorms - OLD}
Let $p>Q$, $0<\eps<1/8$, and $\alpha_p=1-\frac Qp$. Let
$K\subset X$ be compact, and let $\Omega\subset X$ be open with $K\subset \Omega$. Assume that $U\subset X$ is a nonempty measurable set with
$$
\lambda\diam U<\dist(U,X\setminus K), 
$$
where $\lambda$ is the Poincar\'e radial constant.
Then there exists a positive constant $C$ such that every continuous $f:X\to V$ satisfies
$$
    C^{-1}\|f\|_{\dot{M}^{1,p}(U)}
    \leq
    [f]_{p,\alpha_p,\eps}(K)
    \leq
    C\|f\|_{\dot{M}^{1,p}(\Omega)}.
$$
\end{theorem}

\section{Example of Sobolev not in modified compactly H\"older}
\begin{theorem}\label{thm: W1p not CH}
    For any $n\in \N$ and $p>n$, there is a continuous  $u:\R^n\to \R$
such that $u\in W^{1,p}(\R^n)\subset CH^{p, 1-n/p}(\R^n,\R)$,
but $u\notin CH_m^{p, 1-n/p}(\R^n,\R)$.
\end{theorem}

\begin{proof}

Let $n\in \N$, $p>n$, and $\alpha:=1-n/p$.
Fix a large $M\ge 100$. Let $e_1:=(1,0,\dots,0)$, and set
$$
    q_j:=M^{-j^2} e_1,
    \qquad
    s_j:=(M^{-j^2})^2=M^{-2j^2}.
$$
Choose $\phi\in C_c^\infty(B(0,1/100))$ so that $\phi(0)=0$ and $\phi(x_0)=1$, for some point $x_0\in B(0,1/100)$.

Define
$$
    u_j(x):=j^{-2/p} s_j^\alpha
    \phi\left(\frac{x-q_j}{s_j}\right),
$$
and set
$$
    u(x):=\sum_{j=1}^\infty u_j(x).
$$

Note that
$$
\operatorname{supp}(u_j)\subset B(q_j,s_j/100),
$$
and since $s_j/100+s_{j+1}/100<M^{-j^2}-M^{-(j+1)^2}$, the balls $B(q_j,s_j/100)$
are pairwise disjoint, after increasing $M$ if necessary.
Hence,  at every point of
$\R^n$, at most one term of $u$ is nonzero. 

We now prove that $u\in W^{1,p}(\R^n)$. Since the supports are pairwise
disjoint,
$$
\|u\|_{L^p(\R^n)}^p=\sum_{j=1}^\infty \|u_j\|_{L^p(\R^n)}^p
$$
and, since $\nabla u$ exists in $\R^n\setminus\{0\}$, we have
$$
\|\nabla u\|_{L^p(\R^n)}^p=\sum_{j=1}^\infty \|\nabla u_j\|_{L^p(\R^n)}^p.
$$ 
By a change of variables $y=(x-q_j)/s_j$ and noting that $\alpha p+n=p$ we have
$$\|u_j\|_{L^p(\R^n)}^p=j^{-2} s_j^{\alpha p+n}
    \|\phi\|_{L^p(\R^n)}^p=j^{-2} s_j^p\|\phi\|_{L^p(\R^n)}^p,
$$ for all $j\in \N$. Since  $\sum_{j=1}^\infty j^{-2} s_j^p=\sum_{j=1}^\infty j^{-2} (M^{-2j^2})^p<\infty$, it follows that $\|u\|_{L^p(\R^n)}<\infty$

Similarly,
$$\nabla u_j(x)=j^{-2/p} s_j^{\alpha-1}\nabla\phi\left(\frac{x-q_j}{s_j}\right),
$$
and 
$$\|\nabla u_j\|_{L^p(\R^n)}^p=j^{-2} s_j^{p(\alpha-1)+n} \|\nabla\phi\|_{L^p(\R^n)}^p.
$$
But due to $p(\alpha-1)+n=0$, the above implies
$$\|\nabla u_j\|_{L^p(\R^n)}^p=j^{-2}\|\nabla\phi\|_{L^p(\R^n)}^p,
$$
which is enough to show that
$$\|\nabla u\|_{L^p(\R^n)}^p=\sum_{j=1}^\infty \|\nabla u_j\|_{L^p(\R^n)}^p<\infty.
$$
Thus, $u\in W^{1,p}(\R^n)$.

Moreover, $u$ is continuous on $\R^n$. Indeed, let $x\in \R^n$ and $x_k\in \R^n$ for $k\in \N$ with $x_k\rightarrow x$. If $x\in B(q_j,s_j/100)$ for some $j\in \N$, then $u$ is continuous at $x$ because $u_j$ is smooth. If $x\notin \overline{\cup_j B(q_j,s_j/100)}$, then $x_k$ must eventually lie outside the support of $u$, implying that $u(x_k)=0=u(x)$ for all $k$ large enough. Thus, the only case left is $x\in \overline{\cup_j B(q_j,s_j/100)}\setminus (\cup_j B(q_j,s_j/100))$. If $x\neq 0$, then it would lie on the boundary of some $B(q_j,s_j/100)$, which due to $\operatorname{supp}(u_j)$ being compactly contained in $B(q_j,s_j/100)$ yields similarly that $u(x_k)=0=u(x)$ for all $k$ large enough. If $x=0$, then 
$$
|u(x_k)|\leq j_k^{-2/p}s_{j_k}^\alpha \|\phi\|_\infty\rightarrow 0,
$$ as $k\rightarrow \infty$ because $u(x_k)\neq 0$ implies that $x_k\in B(q_{j_k},s_{j_k}/100)$ for progressively smaller balls.

We now show that $u\notin CH_m^{p, \alpha}(\R^n,\R)$.
Set $E=\{0\}$ compact, and fix $\varepsilon=1/100$.
For integers $1\le k\le j$, define
$$
q_{j,k}:=M^{k-j^2}e_1,\qquad\rho_{j,k}:=2M^{k-j^2},
$$ and let $B_{j,k}:=B(q_{j,k},\rho_{j,k})$.
Each $B_{j,k}$ contains $0$, because $|q_{j,k}|=M^{k-j^2}<2M^{k-j^2}=\rho_{j,k}$, and,
thus $\{B_{j,k}\}$ covers the compact set $E=\{0\}$.

Note that every $B_{j,k}$  contains both $q_j$ and $q_j+s_jx_0$. Indeed,
$$
|q_j-q_{j,k}|=|M^{-j^2}-M^{k-j^2}|=M^{k-j^2}-M^{-j^2}<M^{k-j^2}<2M^{k-j^2}=\rho_{j,k},
$$
and due to $|x_0|M^{-2j^2}\leq M^k M^{-j^2}$ we also have
$$|q_j+s_jx_0-q_{j,k}|\le|q_j-q_{j,k}|+s_j|x_0|<M^{k-j^2}+s_j|x_0|\leq 2 M^{k-j^2}=\rho_{j,k}.
$$
Since $q_j, q_j+s_jx_0\in B_{j,k}$, using these two points in the definition of the $\alpha$-H\"older coefficient gives
\begin{equation}\label{eq: low bd holder1}
    |u|_{\alpha,B_{j,k}}\ge\frac{|u(q_j+s_jx_0)-u(q_j)|}{|s_jx_0|^\alpha}.
\end{equation}
Because the supports of the $u_j$'s are pairwise disjoint,
$$u(q_j)=u_j(q_j)=j^{-2/p} s_j^\alpha\phi(0)=0,
$$
and
$$u(q_j+s_jx_0)=u_j(q_j+s_jx_0)=j^{-2/p} s_j^\alpha\phi(x_0)=j^{-2/p} s_j^\alpha.
$$
Therefore, using the above on \eqref{eq: low bd holder1} and raising to the $p$-th yields
\begin{equation}\label{eq: low bd holder2}
    |u|_{\alpha,B_{j,k}}^p\ge\left(\frac{j^{-2/p} s_j^\alpha}{s_j^\alpha|x_0|^\alpha}\right)^p=|x_0|^{-\alpha p}j^{-2}.
\end{equation}

It remains to verify that the shrunken balls are disjoint. Let $j,k,l,m$ be integers such that $1\leq k\leq j$, $1\leq l\leq m$, and $(j,k)\neq(m,l)$. In order to show that
$$
B(q_{j,k}, \varepsilon \rho_{j,k})\cap B(q_{m,l}, \varepsilon \rho_{m,l})=\emptyset,
$$ it is enough to show that
\begin{equation}\label{eq: shr balls disjoint equiv ineq}
    |M^{l-m^2}-M^{k-j^2}|>2\varepsilon (M^{l-m^2}+M^{k-j^2}).
\end{equation}
Without loss of generality assume that $M^{l-m^2}>M^{k-j^2}$, and the proof is similar in the other case. This implies that $l-m^2>k-j^2$, and since $j,k,l,m\in \N$, we have that $l-m^2-(k-j^2)\geq 1$. This and the choices $\varepsilon=1/100$ and $M\geq 100$ yield
$$
M^{l-m^2}\geq M \,M^{k-j^2}>  \frac{51}{49} M^{k-j^2},
$$which implies that \eqref{eq: shr balls disjoint equiv ineq} is indeed true. Thus, $B(q_{j,k}, \varepsilon \rho_{j,k})\cap B(q_{m,l}, \varepsilon \rho_{m,l})=\emptyset$ as needed.

Assume toward a contradiction that $u\in CH_m^{p, \alpha}(\R^n,\R)$. Apply the definition to the compact set
$E=\{0\}$ and the choice $\varepsilon=1/100$. Then there exist constants $ r_E, C_E>0,$
such that every  $\varepsilon$-separated cover of $E$ by balls
of radius less than $r_E$ satisfies
$$
\sum_i |u|_{\alpha,B_i}^p\le C_E.
$$

But
$$\rho_{j,k}=2M^{k-j^2}\le2M^{j-j^2}\to 0
    \qquad\text{as }\,\,j\to\infty.
$$
Thus, we can choose $J$ large enough for $\rho_{j,k}<r_E$ for all $j\ge J,\ 1\le k\le j$.
For $N>J$, consider the finite collection
$$\mathcal B_{J,N}:=\{B_{j,k}: J\le j\le N,\ 1\le k\le j\}.
$$
This is an admissible cover of
$E=\{0\}$, and all of its balls have radius less than $r_E$. Therefore the modified compactly
H\"older condition implies
$$
\sum_{j=J}^N\sum_{k=1}^j |u|_{\alpha,B_{j,k}}^p\le C_E.
$$
However, by \eqref{eq: low bd holder2} we have
$$
C_E\geq \sum_{j=J}^N\sum_{k=1}^j |u|_{\alpha,B_{j,k}}^p\ge |x_0|^{-\alpha p}\sum_{j=J}^N\sum_{k=1}^j \frac1{j^2}=|x_0|^{-\alpha p}
    \sum_{j=J}^N \frac1j.
$$
The right-hand side tends to $\infty$ as $N\to\infty$, contradicting the
existence of $C_E$. Therefore, $u\notin CH_m^{p, \alpha}(\R^n,\R)$.
\end{proof}

\bibliographystyle{acm}
\bibliography{SobolevDim}
\end{document}